\documentclass[11pt]{article}
\usepackage{latexsym,amsmath,amsthm,verbatim,ifthen,amssymb}
\usepackage[latin1]{inputenc}
\usepackage[all]{xy}
\usepackage{graphicx}
\usepackage{epsfig}
 \usepackage{color}
\newtheorem*{theorem*}{Theorem}
\newdir{ >}{{}*!/-12pt/@{>}}
\newdir{ (}{!/-3pt/@_{(} } 

\newdir{ )}{!/-3pt/@^{(} } 

\newcommand{\ex}{\mathcal{E}} 
\newcommand{\T}{\tau}  

\newcommand{\E}{\textbf{E}}

\newcommand{\R}{\mathbb{R}}

\newcommand{\N}{\mathbb{N}}
\newcommand{\fib}{\operatorname{{{\rm Fib}}}_F}

\newcommand{\extbas}{\mathbf{E}^{\mathbb{R}_+}}

\newcommand{\SN}{\mathfrak{S}}

\newcommand{\DN}{\mathfrak{D}}

\newtheorem{theorem}{Theorem}[section]

\newtheorem{corollary}[theorem]{Corollary}
 \newtheorem{lemma}[theorem]{Lemma}
 \newtheorem{proposition}[theorem]{Proposition}
 \theoremstyle{definition}
 \newtheorem{definition}[theorem]{Definition}
 \theoremstyle{remark}

 \numberwithin{equation}{subsection}

\newcommand{\br}{\mathbb R}

\newcommand{\PP}{\textbf{P}}

\begin{document}

\title{Classification of exterior and proper fibrations\footnotetext{ This work has been supported by
the MICINN grant MTM2016-78647 of the Spanish Government and
 and by the Junta de Andaluc\'{\i}a grant FQM-213.\vskip 4pt
2010 \textit{Mathematics Subject Classification}: 55P57, 55R05, 55R15.
\vskip 4pt
\textit{Keywords}: Proper homotopy theory, exterior homotopy theory, exterior fibration, proper fibration, Brown
representability.}}

\author{J. M. Garc\'{\i}a-Calcines, P. R. Garc\'{\i}a-D\'{\i}az \\
and A. Murillo}

\maketitle

\begin{abstract}
We classify  exterior fibrations in the exterior homotopy category. As a result we also classify proper fibrations between CW-complexes.
\end{abstract}

\section*{Introduction}
It is known that the absence of most usual categorical properties, such as the existence of (co)limits, constitutes the main handicap for the development of proper homotopy theory. An illustrative example is given by the fact that the category  $\PP$  of topological spaces and proper maps is an {\em $I$-category} \cite{Ay-D-Q}, and thus it is also a {\em cofibration category}, while even the definition of what a proper fibration should be presents some issues \cite{adoquin,chap}. However, $\PP$ is embedded in the category $\E$ of exterior spaces which is complete and cocomplete and it has model category structures \cite{Ext_1,Ext_2,G-G-M} closely related to the classical ones on topological spaces (see \S1 for notation and terminology). From this, the concept of proper fibration becomes clear as  exterior fibration in the proper category. In this paper we  classify exterior fibrations obtaining as an immediate consequence also the classification of proper fibrations.

On the free setting, we may consider any of these available  model structures on $\E$ and apply directly the Blomgren and Chach\'olski classification of fibrations on model structures \cite[Thm. B]{blomcha}, inspired in the foundational works of May  \cite{may} and Stasheff \cite{s}. As a result one obtains that the whole moduli space of fiber exterior homotopy equivalences over an exterior space $X\in \E$ with fiber $F\in \E$ is weakly equivalent to a mapping space with target the classifying space of the monoid of exterior self weak equivalences of $F$.

On the other hand, the classification of exterior fibrations in the based setting is not straightforward. Indeed, general statements of Brown representability, see for instance \cite[\S3]{jardine}, cannot be applied to the based exterior category $\extbas$ as this is no longer a model category. Thus, as in \cite{al,schon}, we attack the classification based on the original work of Brown \cite{Br} to prove the following (see Theorem \ref{main} for a precise statement).

 Let $F$ be an exterior path connected based  CW-complex and denote by $\fib(X)$ the set of equivalence classes of exterior fiber sequences over  a given exterior path connected based CW-complex $X$ with fiber $F$.

\begin{theorem*} There exists an exterior path connected based CW-complex  $Y_F$, unique up to exterior homotopy, such that,
$$
\fib(X)\cong [X,Y_F]^{\br_+}.
$$
\end{theorem*}
As a result, see Corollary \ref{proper}, we exhibit any based proper fibration $p\colon (E,\br_+)\to (X,\br_+)$    between countable, locally finite relative CW-\break complexes, with fiber $F$, as the exterior pullback of the universal fibration $F\to U\to Y_F$ in which $Y_F$ lies  along a based exterior map $f\colon X\to Y_F$.

\section{Preliminaries}

We shall be using the following standard facts on  exterior homotopy theory whose summary can be found  in \cite[\S2]{G-G-M2} or \cite[\S1]{G-G-M3}.

An \textit{exterior space}
$(X,\ex)$ is a topological space $(X,\T)$ endowed with an \textit{externology} $\ex\subset \T $, i.e.,   a
non empty family of so called {\em exterior sets} which can be thought as a neighborhood system at infinity: it is closed under finite intersections
and, whenever $U \supset E$, $E \in \ex$, $U\in \T$, then $U\in
\ex$.
 An  \textit{exterior} map $f\colon (X,\ex ) \rightarrow (X',\ex'
)$ is a continuous map for which $f^{-1}(E) \in
\ex$, for all $E \in \ex'$. The cylinder functor in this category assigns to each exterior space $X$ the topological space $X\times I$ in which an open set is exterior if it contains some $E\times I$ with $E$ exterior set of $X$. To avoid confusion with the usual externology on the product of two exterior spaces, we denote by $X\bar\times I$ this exterior space. Exterior homotopy is defined accordingly.

The \textit{cocompact externology} $\ex_{cc}$ on a given topological space $X$  is formed by the
family of the complements of all closed-compact sets of $X$. Denote by $X_{cc}$ the
corresponding exterior space. This defines a full embedding
\cite[Thm. 3.2]{Ext_1}
\begin{equation}\label{embebe}
(-)_{cc}\colon \PP \hookrightarrow \E
\end{equation}
from the proper category $\PP$ of topological spaces and proper maps. As $(X\times I)_{cc}=X_{cc}\bar\times I$ this embedding extends to the respective homotopy categories.

Unlike $\PP$, the category $\E$ is complete, cocomplete
\cite[Thm. 3.3]{Ext_1} and it has a closed model structure in which fibrations, cofibrations and weak equivalences are respectively  the exterior maps satisfying the homotopy lifting property, the exterior closed maps satisfying the homotopy extension property, and the exterior homotopy equivalences respectively \cite[Thm. 2.10]{G-G-M}.

In the based setting, and both in $\PP$ and $\E$, the ``point'' is the  half-line $\mathbb{R}_+=[0,\infty)$ endowed with the cocompact externology. Accordingly, the objects of the category $\E^{\R_+}$ of {\em well pointed exterior spaces}\footnote{We warn the reader that this category was denoted by $\E^{\R_+}_w$ in  \cite{G-G-M2,G-G-M3} while $\E^{\R_+}$ was reserved for based exterior spaces, non necessarily well pointed, i.e., the based ray is not necessarily a closed cofibraton. As all based exterior spaces we consider here are well pointed we  avoid excessive notation.}   are pairs $(X,\alpha)$, or simply $(X,\br_+)$, in which $X\in \E$ and $\alpha\colon \mathbb{R}_+\to X$ is a closed exterior cofibration called the {\em based ray}.  Morphisms $f\colon (X,\alpha ) \rightarrow
(Y,\beta )$ are exterior maps $f\colon X\rightarrow Y$ for which
$f \alpha = \beta $. Homotopy in $\E^{\R_+}$ is defined through the functor  which assigns to each
$(X,\br_+ )\in \E^{\R_+}$, the \textit{based exterior cylinder of
$X$}, $I^{\R_+} X$,  defined by the pushout:
$$
\xymatrix{
  \R_+\bar\times I \ar[d] \ar[r] & \R_+ \ar[d]^{} \\
  X\bar\times I \ar[r]^{} & I^{\R_+} X.  }
$$
Then, $\E^{\R_+}$ verifies all the axioms of a closed
model category, except for being closed for finite limits and
colimits  \cite[Thm. 2.12]{G-G-M}. Fibrations (resp. cofibrations)
are  exterior based maps which verify the Homotopy Lifting
Property (resp. the Homotopy Extension Property), while  weak
equivalences are  based exterior homotopy equivalences.
Nevertheless, pullbacks of fibrations and
pushouts of cofibrations  can be constructed within $\E^{\R_+}$
and thus, {\em exterior homotopy pullbacks and pushouts} are
defined in this category. We denote by $\mathbf{Ho}\extbas$ the corresponding homotopy category with
the same objects, and whose set of morphisms $[X,Y]^{\R_+}$
between $(X,\R_+)$ and $(Y,\R_+)$ are based exterior homotopy classes of exterior maps.

Finally, a  space $X\in\extbas$ is {\em exterior path connected} if it is path connected as a topological space and  $[S^0_+,X]^{\R_+}=\{*\}$. Here, $S^0_+$ denotes $\br_+$ with a $0$-sphere
attached to each integer number,  endowed with the cocompact
externology, and the obvious ray
$\br_+\hookrightarrow  S^0_+$. For instance, in the based proper category $\PP^{\br_+}$, and under very mild conditions, a space $X$ is exterior path connected if and only if it has only one Freudenthal end.

Next, recall the notion of {\em exterior CW-complexes} which include, in the proper
case,   spherical objects under a tree
\cite[IV]{B-Q}.
From now on $\N\subset \R_+$ will always be endowed with the
induced externology.
Given $n\ge 0$, we denote by  $\SN^k$  either the
sphere $S^k$ or the
 $\N$-sphere defined as the exterior space $\N\bar\times S^{k}$. Analogously $\DN^k$
will ambiguously denote either the  disk $D^k$
or  the  $\N$-disk $\N\bar\times D^{k}$. The inclusion
 $\SN^{k-1} \hookrightarrow \DN^{k}$ is a closed exterior cofibration.

A \textit{relative exterior CW-complex} $(X,A)$ is an exterior
space $X$ together with a filtration of exterior subspaces
$$A=X^{-1} \subset X^0 \subset X^1 \subset \ldots \subset
X^n\subset \ldots \subset X$$ \noindent for which
$X=\mbox{colim}\hspace{3pt}X^n,$ and for each $n\geq 0$, $X^n$ is
obtained from $X^{n-1}$ as the exterior pushout
$$\xymatrix@C=1.6cm@R=1cm{
  \amalg_{\gamma \in \Gamma} \SN^{n-1} \ar@{^{ (}->}[d]_{} \ar[r]^{\amalg_{\gamma \in \Gamma}\varphi_\gamma} & X^{n-1} \ar@{^{ (}->}[d]^{} \\
  \amalg_{\gamma \in \Gamma} \DN^{n} \ar[r] & X^n   }$$
via the  attaching maps $\varphi_\gamma\colon \SN^{n-1} \rightarrow
X^{n-1}$ and where $\amalg$ denotes disjoint union. When $A=\emptyset$ (respec. $A=\mathbb{R}_+$) we recover the notion of
 \emph{exterior CW-complex} (respec.
\emph{based exterior CW-complex}). In the last case $(X,\alpha)$  is necessarily well based as
the inclusion
$\mathbb{R}_+\hookrightarrow X$ is a closed cofibration).

Remark that a finite exterior CW-complex is in general a finite
dimensional infinite classical CW-complex. Also,  any classical
CW-complex is an exterior CW-complex with its topology as
externology. Other important class of exterior CW-complexes are
constituted by the open  manifolds and PL-manifolds as
they admit a locally finite countable triangulation, which
describes the exterior CW-structure \cite[\S 2(ii)]{G-G-M2}.
We denote by $\mathbf{CW}^{\mathbb{R}_+}$, (resp.
$\mathbf{CW}^{\mathbb{R}_+}_f$)  the full subcategory of
$\mathbf{E}^{\mathbb{R}_+}$ formed by based exterior
CW-complexes (resp. finite based exterior CW-complexes), and
by $ \mathbf{Ho CW}^{\mathbb{R}_+}$ (resp. $ \mathbf{Ho
CW}^{\mathbb{R}_+}_f$) the corresponding homotopy categories.

Finally, recall  that  a map $f\colon X\to Y$ between
exterior path connected spaces in  $ \mathbf{CW}^{\mathbb{R}_+}$ is an exterior based
homotopy equivalence if and only if $f_*\colon [Z,X]^{\R_+}\to
[Z,Y]^{\R_+}$ is a bijection for every exterior path connected $Z\in
\mathbf{CW}^{\mathbb{R}_+}_f$.

\section{Classification of exterior and proper fibrations}

To avoid excessive terminology every CW-complex considered henceforth will be exterior, based and exterior path connected unless explicitly stated otherwise.

In particular, also for simplicity in the  notation, we abuse of it and let $ \mathbf{Ho
CW}^{\mathbb{R}_+}$ (respec. $ \mathbf{Ho
CW}^{\mathbb{R}_+}_f$) denote the  homotopy category of exterior path connected,  based (respec. finite) CW-complexes.

Let $F\in\extbas$ a well based exterior space. Define a contravariant functor
$$
\fib\colon\mathbf{Ho}\extbas\longrightarrow \mathbf{Sets}
$$
as follows:

For each   $X\in\extbas$, $\fib(X)$ is the set of equivalence classes of {\em exterior fiber sequences over $X$ with fiber $F$}.  These are  sequences $\mathcal{F}$ in  $\extbas$ of the form
$$F\stackrel{g}{\longrightarrow }E\stackrel{p}{\twoheadrightarrow}X$$
where
$p\colon E\twoheadrightarrow X$ is an exterior fibration
in $\extbas$ and there is an exterior homotopy pullback
$$\xymatrix{
{F} \ar[d]_{r} \ar[r]^g & {E} \ar@{>>}[d]^p \\
{\mathbb{R}_+} \ar[r]_{\alpha _X} & {X} }$$
where $\alpha_X$  is the based ray of $X$. That is, there exists a commutative

$$\xymatrix@C=.7cm@R=.5cm{
{F} \ar[dd]_{r} \ar[rr]^{g} \ar@{.>}[dr]^{\simeq} & &  {E} \ar[dd]^p\\
 & {\bullet} \ar[dl]\ar[ur]  & \\
 {\br_+} \ar[rr]_{\alpha_X}  & & X  }$$
where the inner square is the pullback of $p$ along $\alpha_X$ and the dotted induced map is an exterior homotopy equivalence.

 Note that this implies the existence of the map $r\colon F\to\br_+$ which is necessarily unique up to exterior homotopy \cite[Rem. 2.18]{G-G-M2},  and therefore, it is an exterior homotopy retraction of the ray of $F$.  Two fiber sequences $F\stackrel{g_1}{\longrightarrow }E_1\stackrel{p_1}{\twoheadrightarrow}X$ and $F\stackrel{g_2}{\longrightarrow }E_2\stackrel{p_2}{\twoheadrightarrow}X$ are equivalent if there exists a homotopy commutative diagram
in $\mathbf{E}^{\mathbb{R}_+}$ of the form
$$\xymatrix{
 & {E_1} \ar@{>>}[dr]^{p_1} \ar[dd]_{\simeq }^{\gamma } & \\
 {F} \ar[ur]^{g_1} \ar[dr]_{g_2} & & {X} \\
 & {E_2} \ar@{>>}[ur]_{p_2} &
}$$ \noindent with $\gamma $ an exterior based homotopy
equivalence.

On the other hand, given $f\colon Y\rightarrow X$ a (homotopy class of a) based exterior map
$\fib(f)$ associates to $F\stackrel{g}{\longrightarrow }E\stackrel{p}{\twoheadrightarrow}Y$ the fiber sequence $F\stackrel{g'}{\longrightarrow }E\stackrel{p'}{\twoheadrightarrow}Y$ in which $p'$ is obtained as the based exterior pullback of $p$ along $f$ and $g'$ is induced by $g$ and $\alpha_Y\circ r$:
$$\xymatrix{
  {F} \ar@{.>}[dr] \ar@/_/[ddr]_{\alpha _Yr} \ar@/^/[drr]^g
                       \\
   & {E'}  \ar@{>>}[d]^{p'} \ar[r]^{f'}
                      & E \ar@{>>}[d]^p    \\
   & {Y} \ar[r]_{f}     & X               }
$$

We prove:

\begin{theorem}\label{main} The restriction,
$$
\fib\colon\mathbf{Ho CW}^{\mathbb{R}_+}\longrightarrow \mathbf{Sets}
$$
is a representable functor: there is a  based exterior path connected CW-complex $Y_F$, unique up to based exterior homotopy, and a {universal} based exterior fiber sequence
$$
F\rightarrow U\stackrel{q}{\twoheadrightarrow
}Y_F
$$
such that
$$
[-,Y_F]^{\mathbb{R}_+}\stackrel{\simeq}{\longrightarrow} \fib\qquad f\mapsto \fib(f)(q)$$  is
a natural equivalence.
\end{theorem}

The rest of the section is devoted to the proof of this theorem. The requirements for a  set-valued  contravariant functors on the homotopy
category of a given closed model category to satisfy  Brown's representability  are now  well understood. However, the most explicit and precise result in this sense
\cite[Thm. 19]{jardine} cannot be applied in our case as  $\E^{\R_+}$ is not a model category. Hence, we use the original Brown approach \cite{Br}.

We shall need the following auxiliary results of general nature: the
\emph{First Cube Theorem} \cite[Thm. 18]{M} and the {\em Gluing Lemma} in $\extbas$. The proofs just mimics the ones on the classical setting  and therefore are omitted.  To
adjust them to the exterior homotopy setting, certain
modifications of a somehow straightforward nature are needed.

\begin{lemma}\label{cube-theorem}
Consider the following homotopy commutative cube in
$\mathbf{E}^{\mathbb{R}_+}$
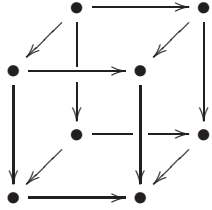
$$\xymatrix@!0{
  & {\bullet } \ar[dl] \ar[rr] \ar'[d][dd]
      &  & {\bullet } \ar[dd] \ar[dl]       \\
  {\bullet } \ar[rr]\ar[dd]
      &  & {\bullet } \ar[dd] \\
  & {\bullet } \ar[dl] \ar'[r][rr]
      &  & {\bullet }   \ar[dl]             \\
  {\bullet } \ar[rr]
      &  & {\bullet }         }$$
\noindent where the top and bottom faces are exterior homotopy
pushouts and the left and rear faces are exterior homotopy
pullbacks. Then, the right and front faces are exterior homotopy
pullbacks.\hfill$\square$
\end{lemma}

\begin{lemma} \label{gluing lemma} Consider the following homotopy commutative cube in
$\mathbf{E}^{\mathbb{R}_+}$
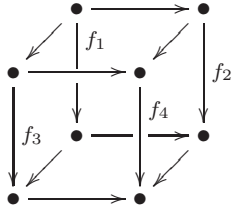
$$\xymatrix@!0{
  & {\bullet } \ar[dl] \ar[rr] \ar'[d]^{f_1}[dd]
      &  & {\bullet } \ar[dd]^{f_2} \ar[dl]       \\
  {\bullet } \ar[rr]\ar[dd]^{f_3}
      &  & {\bullet } \ar[dd]^(.3){f_4} \\
  & {\bullet } \ar[dl] \ar'[r][rr]
      &  & {\bullet }   \ar[dl]             \\
  {\bullet } \ar[rr]
      &  & {\bullet }         }$$
\noindent where the top and bottom faces are exterior homotopy pushouts and $f_1,f_2,f_3$ are exterior homotopy equivalences.
Then, $f_4$ is also an exterior homotopy equivalence.\hfill$\square$
\end{lemma}

 With the vocabulary in \cite{Br}, we now see that on  $\mathbf{Ho CW}^{\mathbb{R}_+}$,   $\fib$ is a {\em homotopy functor}.

\begin{proposition}\label{equalizer-condition} The functor $\fib$ takes wedges into products:
Let $\{X_i\}_{i\in I}$ be a collection of objects in $\mathbf{CW}^{\mathbb{R}_+}$ and denote by $h_j\colon X_j\rightarrow \vee
_{i\in I}X_i$ the natural $j$th inclusion, $j\in I$. Then, the map
$$
\bigl(\fib(h_i)\bigr)_{i\in I}\colon \fib(\vee _{i\in I}X_i)\stackrel{\cong
}{\longrightarrow }\Pi_{i\in I}\fib(X_i)
$$
is a bijection.
\end{proposition}

\begin{proof}
Let $\{F\rightarrow E_i\stackrel{p_i}{\rightarrow }X_i\}_{i\in I}$
be a collection of exterior fibrations and consider the
following commutative cube of based exterior CW-complexes,
$$\xymatrix@R=0.90cm@C=1.50cm@!0{
  & {\amalg_{i\in I}\,\R_+\times F} \ar[rrr]
  \ar'[dd][dddd] \ar[ddl]
      &  &  &  {\amalg_{i\in I}E_i} \ar[dddd]^{\amalg_{i\in I}p_i} \ar[ddl] \\
                &  &  & \\
  F \ar[rrr] \ar@{>>}[dddd]_r
      &  &  &  {E} \ar@{.>}[dddd]^(.3){p} \\
                &  &  &   \\
  & {\amalg_{i\in I}\,\R_+} \ar'[rr][rrr]  \ar[ddl] &  &  & {\amalg_{i\in I}X_i}
  \ar[ddl]^{} \\
                &  &  & \\
  {\mathbb{R}_+} \ar[rrr]  &  &  &  {\vee_{i\in I}X_i}   }$$
\noindent where the top and bottom faces are exterior homotopy
pushouts and the rear and left faces are exterior homotopy
pullbacks. Then, taking $p\colon E\rightarrow \vee_{i\in I}X_i$ the
induced based exterior map and applying Lemma \ref{cube-theorem} we conclude that the
right and front faces are exterior homotopy pullbacks. This proves
that $\bigl(\fib(h_i)\bigr)_{i\in I}$ is onto.

Now let
$F\stackrel{g^k}\rightarrow E^k\stackrel{p^k}\twoheadrightarrow
\vee_{i\in I}X_i$ ($k=1,2$) be two based exterior sequences such that, for
every $i\in I$, there exist a  commutative diagram
$$\xymatrix{
 & {E^1_i} \ar@{>>}[dr]^{p^1_i} \ar[dd]_{\simeq }^{\alpha _i} & \\
 {F} \ar[ur]^{g^1_i} \ar[dr]_{g^2_i} & & {X_i} \\
 & {E^2_i} \ar@{>>}[ur]_{p^2_i} &
}$$ \noindent where $F\stackrel{g^k_i}\rightarrow
E^k_i\stackrel{p^k_i}\twoheadrightarrow X_i$ is $\fib(h_i)(p^k)$, i.e,  the
pullback of $p^k$ along $h_i$. Consider the cube
$$\xymatrix@R=0.90cm@C=1.50cm@!0{
  & {\amalg_{i\in I}\,\R_+\times F} \ar[rrr]
  \ar@{=}'[dd][dddd] \ar[ddl]
      &  &  &  {\coprod _{i\in I}E^1_i} \ar[dddd]^{\amalg_{i\in I}\alpha _i}_{\simeq} \ar[ddl] \\
                &  &  & \\
  F \ar[rrr]^{g^1} \ar@{=}[dddd]
      &  &  &  {E^1} \ar@{.>}[dddd]^(.3){\alpha } \\
                &  &  &   \\
  & {\amalg_{i\in I}\, \R_+\times F} \ar'[rr][rrr]  \ar[ddl] &  &  & {\amalg_{i\in I}E^2_i}
  \ar[ddl]^{} \\
                &  &  & \\
  {F} \ar[rrr]^{g^2}  &  &  &  {E^2}   }$$
  where the top and bottom faces are based exterior homotopy pushouts and apply Lemma \ref{gluing lemma} to conclude that the induced map $\alpha$ is a based exterior homotopy equivalence. Moreover, $p^2\alpha \simeq p^1$ so that $\bigl(\fib(h_i)\bigr)_{i\in I}$ is injective.
\end{proof}

\begin{proposition}\label{Mayer-Vietoris}
Let
$$\xymatrix{
{A} \ar[d]_{f_1} \ar[r]^{f_2} & {X_2} \ar[d]^{g_1}  \\
{X_1} \ar[r]_{g_2} & {Y} }$$
be a homotopy pushout in $\mathbf{Ho CW}^{\mathbb{R}_+}$. Then, the induced
map
$$ \fib(Y)\longrightarrow \fib(X_1)\times_{\fib(A)}\fib(X_2)$$
\noindent is surjective: if $\fib(f_1)(p_1)=\fib(f_2)(p_2)$, then there
exists $p\in \fib(Y)$ such that $\fib(g_1)(p)=p_1$ and $\fib(g_2)(p)=p_2$.
\end{proposition}

\begin{proof}
This is just a direct application of Lemma
\ref{cube-theorem}.
\end{proof}

\begin{proof}[Proof of Theorem \ref{main}] If we drop the exterior path connectivity assumption on the category of exterior CW-complexes, the following is proved in  Lemma 4.1 of  \cite{G-G-M3}:

\begin{itemize}
\item[(i)] The category $\mathbf{Ho CW}^{\mathbb{R}_+}$ (respec. $\mathbf{Ho CW}^{\mathbb{R}_+}_f$) has arbitrary (respec. finite) coproducts and homotopy pushouts.

\item[(ii)] There exists the homotopy colimit $Y$  of any direct system
$$\xymatrix{{X_1} \ar[r] & {X_2} \ar[r] & {X_3} \ar[r] \ar[r]
& ... \ar[r] & {X_n} \ar[r] & {X_{n+1}} \ar[r] & ... }$$
in $\mathbf{Ho CW}^{\mathbb{R}_+}$ and the natural maps
$$
[Y,Z]^{\mathbb{R}_+}\twoheadrightarrow
\varprojlim\hspace{3pt}[X_n,Z]^{\mathbb{R}_+}\qquad\text{and}\qquad \varinjlim\hspace{3pt}[Z,X_n]^{\mathbb{R}_+}\stackrel{\cong}{\rightarrow}
[Z,Y]^{\mathbb{R}_+}
$$
are, respectively, a surjection  for every $Z\in \mathbf{Ho CW}^{\mathbb{R}_+}$, and a bijection for every $Z\in \mathbf{Ho CW}^{\mathbb{R}_+}_f$.

\end{itemize}

A careful check shows that the restriction to exterior path connected CW-complexes does not change the assertions above. That is, the arbitrary wedge, the homotopy pushout and the homotopy colimit of any direct system of exterior path connected CW-complexes is also exterior path connected. For it, one needs general results on connectivity and cellular approximation of exterior CW-complexes which, for instance, are condensed in  \cite[\S2]{G-G-M2}.

 For all of the above, the pair  $(\mathbf{Ho CW}^{\mathbb{R}_+},\mathbf{Ho CW}^{\mathbb{R}_+}_f)$ is a {\em homotopy category} in the sense of \cite[\S2]{Br}. Furthermore, tha last paragraph of \S1 amounts to say that $\mathbf{Ho CW}^{\mathbb{R}_+}$ is compactly generated by  $\mathbf{Ho CW}^{\mathbb{R}_+}_f$.

On the other hand, and also with the  vocabulary of op.cit., Propositions \ref{equalizer-condition} and \ref{Mayer-Vietoris} show that $\fib$ is a {\em homotopy functor}.

Hence, applying \cite[Thm. 2.8]{Br} finishes the proof.
\end{proof}

As an application we give a classification of fibrations in the proper setting.

\begin{definition} A {\em proper fibration} is a proper map  $p\colon E\to X$ such that $p_{cc}\colon E_{cc}\to X_{cc}$ is an exterior fibration.
\end{definition}

In other words, a proper fibration is a map in the proper category which is an exterior fibration when considered in the exterior category through the full embedding in (\ref{embebe}).
 This is a slightly different object from that on   \cite[Def. 1]{adoquin}, cf. \cite{chap}, in which the authors consider proper  maps that are  Hurewicz fibrations .

 Let $p\colon (E,\br_+)\to (X,\br_+)$ be a based proper fibration   between countable, locally finite relative CW-complexes, and let $F$ be its fiber regarded as a exterior based space  via the exterior pullback of $\br_+\to X\stackrel{p}{\leftarrow} E$. Then, we have:

 \begin{corollary}\label{proper} The proper fibration $p$ is equivalent to the pullback of the universal fibration $ U\to Y_F$ along a based exterior map $ X\to Y_F$.
\end{corollary}

\begin{proof} By \cite[\S2.1]{G-G-M2} or \cite[\S5.B]{Ext_1} any countable, locally finite
relative CW-complex of the form $(X,\R_+)$ is a based exterior CW-complex
endowed with the cocompact externology. Hence, $F\to E\stackrel{p}{\to} X$  is an exterior fiber sequence  to which we may apply Theorem \ref{main}.
\end{proof}

\bigskip
\bigskip\bigskip
\noindent{\sc J. M. Garc\'{\i}a-Calcines, P. R. Garc\'{\i}a-D\'{\i}az\hfill\break Departamento de Matem\'aticas, Estad{\'\i}stica e I.O.,  Universidad de La Laguna, Ap. 456, 38200 La Laguna, Spain.}\hfill\break
\texttt{jmgarcal@ull.es, prgdiaz@ull.es}
\hfill\break

\medskip
\noindent {\sc A. Murillo \hfill\break Departamento de \'Algebra, Geometr\'{\i}a y Topolog\'{\i}a, Universidad de M\'alaga, Ap.\ 59, 29080 M\'alaga, Spain}.\hfill\break
\texttt{aniceto@uma.es}

\end{document}